\documentclass[10pt,twoside]{amsart}

\usepackage[T1]{fontenc}
\usepackage[utf8]{inputenc}
\usepackage{stix2}
\usepackage[final]{microtype}
\usepackage[
  paperwidth=171mm,
  paperheight=246mm,
  left=40.4pt,
  right=40.5pt,
  top=48pt,
  bottom=49pt,
  headheight=9pt,
  headsep=10pt,
  footskip=26pt,
  marginparsep=0pt,
  marginparwidth=0pt
]{geometry}

\usepackage{graphicx}
\usepackage{mathtools}
\usepackage{amsthm,mathrsfs}

\numberwithin{equation}{section}
\allowdisplaybreaks

\newtheorem{proposition}{Proposition}[section]
\newtheorem{theorem}[proposition]{Theorem}
\newtheorem{lemma}[proposition]{Lemma}
\newtheorem{corollary}[proposition]{Corollary}

\theoremstyle{definition}
\newtheorem{definition}[proposition]{Definition}
\newtheorem{example}[proposition]{Example}

\theoremstyle{remark}
\newtheorem{remark}[proposition]{Remark}

\DeclareMathOperator{\sing}{Sing}

\newcommand{\C}{\mathbb{C}}
\newcommand{\pn}{\mathbb{P}^n}
\newcommand{\pnt}{\widetilde{\mathbb{P}}^n}

\newcommand{\dps}{\displaystyle}
\newcommand{\ep}{\varepsilon}

\newcommand{\kk}{\mathcal{K}}

\newcommand{\fol}{\mathcal{F}}
\newcommand{\F}{\mathcal{F}}

\newcommand{\G}{\mathcal{G}}

\newcommand{\nn}{\mathcal{N}}
\newcommand{\oo}{\mathcal{O}}

\newcommand{\W}{\mathbf{W}}
\newcommand{\N}{\mathbf{N}}

\newcommand{\R}{\mathbf{R}}
\newcommand{\E}{\mathbf{E}}
\newcommand{\Z}{\mathbf{Z}}

\title[Milnor number of vector fields]{A lower bound for the Milnor number of vector fields}

\author[M. Corr\^ea]{Maur\'icio Corr\^ea}
\address{Maur\'icio Corr\^ea \\
Universit\`a degli Studi di Bari,
Via E. Orabona 4, I-70125, Bari, Italy
}
\email[M. Corr\^ea]{mauricio.barros@uniba.it, mauriciomatufmg@gmail.com}

\author[G. N. Costa]{Gilcione Nonato Costa}
\address{Gilcione Nonato Costa \\
ICEx -- UFMG \\
Departamento de Matem\'atica \\
Av. Ant\^onio Carlos 6627 \\
30123-970 Belo Horizonte MG, Brazil}
\email[G. N. Costa]{gilcione@mat.ufmg.br}

\author[A. S. Russi]{Alejandra Salamanca Russi}
\address{Alejandra Salamanca Russi \\
ICEx -- UFMG \\
Departamento de Matem\'atica \\
Av. Ant\^onio Carlos 6627 \\
30123-970 Belo Horizonte MG, Brazil}
\email[A. S. Russi]{alejandra.russi.s@gmail.com}

\begin{document}

 \begin{abstract}
We study holomorphic vector fields whose singular locus contains a smooth
positive-dimensional local complete-intersection component. We obtain an
exact global formula for its Milnor--Poincar\'e--Hopf contribution and prove
a sharp local lower bound, valid under holomorphic perturbations, in terms
of the embedded contribution associated with the component. As an
application to gradient vector fields, we derive a lower bound for the
total Milnor number of the isolated critical points arising near a smooth
positive-dimensional critical component; for a morsification, this becomes
a lower bound for the number of Morse critical points converging to that
component. Explicit families show that the bounds are sharp and exhibit
the redistribution of singularities between a fixed neighbourhood of the
component and the hyperplane at infinity in projective compactifications.
\end{abstract}
\maketitle
\section{Introduction}

Let $X$ be a germ of a holomorphic vector field on $\C^n$ whose singular locus contains a smooth complete-intersection subvariety $W\subset\C^n$ of codimension $d\geq2$. Under a holomorphic deformation, isolated singularities may separate from $W$, remain in a fixed neighbourhood of it, and converge to $W$ as the parameter tends to zero. After projective compactification, part of the total index may instead be concentrated at singularities on the hyperplane at infinity.

Throughout this article, an embedded singular point associated with $W$ means an isolated singularity produced by a sufficiently small admissible deformation and converging to $W$ as the parameter tends to zero. This perturbative interpretation agrees with the deformation theory of Baum--Bott residues developed by Bracci and Suwa \cite{BracciSuwa}. The limiting count is therefore a residual contribution supported on $W$, not merely an invariant of the reduced set underlying $\sing(\fol)$.

Let $\{X_t\}$ be a holomorphic deformation of $X$, defined for $0<|t|<\varepsilon$, such that $X_t\to X$ as $t\to0$ and $\sing(X_t)$ is finite for every $t\neq0$. We set
\begin{equation}\label{dpnm}
\mu(X_t,\W):=\lim_{t\to0}\sum_{p_i^t\in A_{\W}}\mu(X_t,p_i^t),
\end{equation}
where
$
A_{\W}:=\bigl\{p_i^t\in\sing(X_t)\mid p_i^t\to\W\text{ as }t\to0\bigr\},$
and $\mu(X_t,p_i^t)$ denotes the Milnor number, equivalently the Poincar\'e--Hopf index, of $X_t$ at $p_i^t$. Although the value in \eqref{dpnm} depends on the deformation, Theorem~\ref{theorem2} establishes a uniform lower bound determined by the embedded contribution along $W$. Equivalently, it answers the following question affirmatively:

\begin{center}
\textbf{Question.} \textit{Does there exist a lower bound $\mu_0(X,W)$ such that
$\mu(X_t,\W)\geq\mu_0(X,\W)$ for every deformation $\{X_t\}$ as above?}
\end{center}

The case of a positive-dimensional singular set requires arguments that do not arise in the isolated case. If $X$ has an isolated zero, finite determinacy reduces the local study to a finite jet, and the Milnor number is the multiplicity of that zero. If, however, $\sing(X)$ contains a positive-dimensional component $W$, a perturbation may produce several isolated singularities near $W$. These singularities may move within a fixed neighbourhood of $W$, coalesce with $W$ as the parameter tends to zero, or contribute on the hyperplane at infinity after projective compactification. The problem is therefore to determine and control the total contribution of the singularities converging to $W$.

\noindent
Soares \cite{MGS} established sharp upper bounds for the Poincar\'e--Hopf index of a holomorphic vector-field germ $X:(\C^n,0)\to(\C^n,0)$ with an isolated zero at the origin. His proof uses $\mathcal K$-equivalence to replace $X$ by a polynomial vector field $Y$ of suitable degree, still with an isolated zero at $0\in\C^n$. Since the Poincar\'e--Hopf index is invariant under $\mathcal K$-equivalence,
\[
\mathcal I_0(Y)=\mathcal I_0(X).
\]

For a vector field with an isolated zero, Mather's finite-determinacy theorem implies that the germ $X:(\C^n,0)\to(\C^n,0)$ is $\mathcal K$-equivalent to a finite jet $\mathbf J_0^kX$. Thus $X$ may be replaced, up to $\mathcal K$-equivalence, by the polynomial vector field $\mathbf J_0^kX$, which has an isolated zero at $0$ and satisfies
\[
\mathcal I_0(\mathbf J_0^kX)=\mathcal I_0(X).
\]

The polynomial vector field $\mathbf J_0^kX$ induces a one-dimensional singular holomorphic foliation on $\mathbb P^n$. After an analytic perturbation which removes any positive-dimensional component of the singular set, the Baum--Bott theorem relates the global invariants of this foliation to the local index at the origin. Here $\mathcal I_0(X)$ denotes the Poincar\'e--Hopf index of $X$ at $0$. Soares proved the following theorem.

\begin{theorem}[\cite{MGS}]\label{ms1}
Let $X = X_1\frac{\partial}{\partial z_1}+\cdots+X_n\frac{\partial}{\partial z_n}$
be a germ of a holomorphic vector field at $0 \in \C^n$ such that $0$ is an isolated
zero of $X$, and let $k$ be the degree of $\mathcal K$-determinacy of $X$.
Then the following assertions hold.
\begin{enumerate}
\item[(i)] If $k=1$, then $\mathcal I_0(X)=1$.
\item[(ii)] If $k>1$ and $\mathbf{J}_0^kX= gR+\sum_{j=1}^{k-1}Y_j$, where
$R=\sum_{i=1}^{n}z_i\frac{\partial}{\partial z_i}$ is the radial vector field,
$g\in\C[z_1,\ldots,z_n]$ is homogeneous of degree $k-1$, and each $Y_j$ has homogeneous
components of degree $j$, then
\[
\mathcal I_0(X)\leq \sum_{i=1}^{n}(k-1)^i.
\]
\item[(iii)] If $k>1$ and $\mathbf{J}_0^kX=\sum_{j=1}^{k}Y_j$, where each $Y_j$ has
homogeneous components of degree $j$ and $Y_k$ is not of the form $gR$ (with $g$ and $R$
as above), then
$
\mathcal I_0(X)\leq k^n.
$
\end{enumerate}
\end{theorem}

Esteves and Vainsencher \cite{EV} obtained the same bounds by means of Fulton's intersection theory.

In \cite{GC1,AG}, the corresponding question was considered for one-dimensional holomorphic foliations on $\pn$, with $n\geq3$. Let $\fol$ be such a foliation, of degree $k$, and suppose that its singular set is the disjoint union
\[
\sing(\fol)=\W_0 \,\cup\, \{p_1,\ldots,p_r\},
\]
where
$
\W_0 = Z(f_1,\ldots,f_d)
:=\{z\in \pn \mid f_i(z)=0,\ i=1,\ldots,d\}
$
is a smooth complete intersection subvariety of codimension $d\geq 2$.
Here each $f_i$ is a reduced homogeneous polynomial of degree $k_i$.
Let $T_{\W_0}$ and $\nn_{\W_0}$ denote the tangent and normal bundles of $\W_0$ in $\pn$,
respectively, and let $\pi_1:\pnt\to \pn$ be the blow-up of $\pn$ along $\W_0$,
with exceptional divisor $\E_1=\pi_1^{-1}(\W_0)$.

Assume that $\fol$ is \emph{special along $\W_0$} in the sense of \cite{GC1,AG}. The Milnor number $\mu(\fol,\W_0)$, defined by \eqref{dpnm}, is computed by a holomorphic perturbation $\fol_t$, defined for $0<|t|<\ep$ with $\ep>0$ sufficiently small, having the following properties:
\begin{enumerate}
\item[(i)] $\dps\lim_{t\to 0}\fol_t=\fol$, and $\deg(\fol_t)=\deg(\fol)=k$ for all $t$;
\item[(ii)] for any $t\neq 0$, the vanishing order of $\pi_1^{*}\fol_t$ along $\E_1$ satisfies
\[
m_{\E_1}(\pi_1^*\fol_t)=
\begin{cases}
m_{\E_1}(\pi_1^*\fol), & \text{if $\W_0$ is non-dicritical},\\[2pt]
m_{\E_1}(\pi_1^*\fol)-1, & \text{if $\W_0$ is dicritical};
\end{cases}
\]
\item[(iii)] if $m_{\E_1}(\pi_1^*\fol_t)=0$, then $\W_t$ is $\fol_t$-invariant for all $t\neq 0$,
where $\W_t$ is a holomorphic deformation of $\W_0$ with $\dps\lim_{t\to 0}\W_t=\W_0$;
\item[(iv)] if $m_{\E_1}(\pi_1^*\fol_t)\geq 1$, then
$
\sing(\fol_t)=\W_t \,\cup\, \{p_1^t,\ldots,p_{s_t}^t\},
$
and $\fol_t$ is special along $\W_t$ for all $t\neq 0$.
\end{enumerate}

The formula obtained in \cite{GC1,AG} is
\begin{equation}\label{nmiln}
\mu(\fol,\W_0)=-\nu(\fol,W_0,\varphi_0)+N(\fol_0,A_{\W_0})\geq-\nu(\fol,W_0,\varphi_0).
\end{equation}
Here
\begin{align}\label{nu_equ}
\nu(\fol,W_0,\varphi_0)=-\deg(W_0)\sum_{|a|=0}^{n-d}\sum_{m=0}^{n-d-|a|}(-1)^{\delta_{|a|}^{m}}\frac{\varphi_{a}^{(m)}(\ell)}{m!}(k-1)^{m}\sigma_{a_1}^{(d)}\tau_{a_2}^{(d)}\mathcal{W}_{\delta_{|a|}^{m}}^{(d)},
\end{align}
$k=\deg(\fol_0)$, $a=(a_1,a_2)$, $0\leq a_1 \leq d$, $0\leq a_2 \leq n-d$, $|a|=a_1+a_2$,
$\delta_{|a|}^{m}:=n-d-|a|-m$; and $\ell$ is given by
\[
\ell=
\begin{cases}
 m_{\E_1}\bigl(\pi_1^*\fol_0\bigr), & \text{if $\W_0$ is non-dicritical},\\
 m_{\E_1}\bigl(\pi_1^*\fol_0\bigr)-1, & \text{if $\W_0$ is dicritical}.
\end{cases}
\]
Moreover,
\begin{align*}
\varphi_a(x)&=x^{n-d-a_2}(1+x)^{d-a_1},
& \varphi_a^{(m)}(x)&=\frac{d^m}{dx^m}\varphi_a(x),\\
\mathcal W_\delta^{(d)}
&:=\mathcal W_\delta^{(d)}(k_1,\ldots,k_d)
=\sum_{i_1+\cdots+i_d=\delta}k_1^{i_1}\cdots k_d^{i_d}.
\end{align*}
The term $N(\fol_0,A_{\W_0})$ denotes the number of embedded closed points associated with $W_0$, counted with multiplicity.

It was proved in \cite{AGR} that the Milnor number $\mu(\fol,\W_0)$ defined by \eqref{nmiln} is a \emph{topological invariant}. The preceding formula gives the contribution of each positive-dimensional component of the singular locus.

\begin{theorem}\label{theorem1}
Let $\fol$ be a holomorphic foliation by curves on $\mathbb{P}^{n}$, with $n\geq3$, of degree $k$.
Assume that its singular locus $\sing(\fol)$ is the disjoint union of smooth, non-dicritical, scheme-theoretic complete intersections
$W_1,\ldots,W_r$ of pure codimensions $d_j\geq 2$, together with isolated points $p_1,\ldots,p_s$.
Then the following assertions hold.
\begin{enumerate}
\item[(i)]
\[
\sum_{i=1}^{s}\mu(\fol,p_i)
=
\sum_{i=0}^{n}k^i
+\sum_{j=1}^{r}\nu(\fol,W_j,\varphi_0)
-\sum_{j=1}^{r}N(\fol,A_{W_j}) ,
\]
\item[(ii)]
\[
\mu(\fol,W_j)
:=
N(\fol,A_{W_j})-\nu(\fol,W_j,\varphi_0)
\geq -\,\nu(\fol,W_j,\varphi_0),
\qquad j=1,\ldots,r.
\]
\end{enumerate}
Here $N(\fol,A_{W_j})$ denotes the number of embedded closed points associated with $W_j$, counted with multiplicity.
\end{theorem}

The preceding global formula identifies the contribution supported on a positive-dimensional component. We now establish its local counterpart and prove that the embedded contribution gives a sharp lower bound for every admissible holomorphic perturbation.

\begin{theorem}\label{theorem2}
Let $\fol$ be a germ of a holomorphic foliation by curves on $\mathbb{C}^n$, with $n\geq3$.
Let $\W\subset \sing(\fol)$ be a smooth complete intersection of pure codimension $d\geq2$.
Assume that $\sing(\fol)$ consists of $\W$ together with, possibly, finitely many isolated points.

Let $N(\fol,A_{\W})$ be the number of embedded closed points associated with $\W$ in the perturbative sense described above, counted with multiplicity.
If $N(\fol,A_{\W})$ is finite, then for every deformation $\{\fol_t\}$ of $\fol$, defined for $0<|t|<\epsilon$ with $\epsilon>0$ sufficiently small, such that $\lim_{t\to 0}\fol_t=\fol$ and $\sing(\fol_t)$ consists only of isolated points, one has
\begin{equation}\label{eq:lowerbound}
\mu(\fol_t,\W)
:=
\lim_{t\to 0}\sum_{p_i^t\in A_{\W}^t}\mu(\fol_t,p_i^t)
\geq N(\fol,A_{\W}),
\end{equation}
where
\[
A_{\W}^t
=
\bigl\{\, p\in \sing(\fol_t)\ \big|\ \lim_{t\to 0}p\in \W \,\bigr\}.
\]
Moreover, if $\fol$ is totally simple along $\W$, then there exists a holomorphic perturbation $\fol_t$ of $\fol$ such that
\begin{equation}\label{eq:totallysimplezero}
\mu(\fol_t,\W)
=
\lim_{t\to 0}\sum_{p_i^t\in A_{\W}^t}\mu(\fol_t,p_i^t)
=0.
\end{equation}
\end{theorem}

The local theorem has an immediate consequence for holomorphic function germs with non-isolated critical loci. Let
$
f:(\mathbb C^n,0)\longrightarrow(\mathbb C,0)
$
be a holomorphic germ and set $X=\nabla f$. Then
$
\sing(X)=\operatorname{Crit}(f),
$
and, at every isolated critical point $p$, the Poincar\'e--Hopf index of $\nabla f$ is equal to the Milnor number of $f$ at $p$:
\[
\mathcal I_p(\nabla f)
=
\dim_{\mathbb C}
\frac{\mathcal O_{\mathbb C^n,p}}
{\left(
\frac{\partial f}{\partial z_1},\ldots,
\frac{\partial f}{\partial z_n}
\right)}
=
\mu(f,p);
\]
see \cite{Milnor}. Consequently, Theorem~\ref{theorem2} gives a lower bound for the total Milnor number of the isolated critical points produced near a smooth positive-dimensional component of the critical locus. If the deformation is a morsification, these critical points are non-degenerate, and the same result becomes a lower bound for the number of Morse critical points converging to that component.

\begin{corollary}\label{coro:gradient-introduction}
Let
$
f:(\mathbb C^n,0)\longrightarrow(\mathbb C,0),
\  n\geq3,
$
be a holomorphic germ. Assume that $\operatorname{Crit}(f)$ contains a smooth local complete-intersection component
$
W\subset\operatorname{Crit}(f)
$
of pure codimension $d\geq2$, and that there are at most finitely many isolated critical points away from $W$.

Let $N(\nabla f,A_W)$ denote the embedded contribution associated with $W$ in the perturbative sense described above. If $f_t$ is a holomorphic deformation of $f$ such that $\operatorname{Crit}(f_t)$ is finite for every $t\neq0$, set
\[
A_W^t:=\bigl\{p\in\operatorname{Crit}(f_t)\mid p\to W\text{ as }t\to0\bigr\}.
\]
Then
\[
\lim_{t\to0}\sum_{p\in A_W^t}\mu(f_t,p)
\geq
N(\nabla f,A_W).
\]
In particular, if $f_t$ is a morsification of $f$, then
\[
\lim_{t\to0}\#A_W^t
\geq
N(\nabla f,A_W).
\]
\end{corollary}

The proof is given in Subsection~\ref{subsec:gradient-fields}. Thus the embedded contribution associated with $W$ gives a lower bound both for the total Milnor number arising near $W$ and, in a morsification, for the number of Morse critical points converging to $W$.
For an isolated critical point, the finite-determinacy method used by Soares for holomorphic vector fields applies in particular to the gradient field $\nabla f$ \cite{MGS}. In that case,
\[
\mu_0(f)=\mathcal I_0(\nabla f),
\]
and the resulting estimates belong to the isolated-zero theory. Corollary~\ref{coro:gradient-introduction} concerns the different situation in which the critical locus contains a positive-dimensional component.

The splitting of a positive-dimensional critical locus under deformation is classically described by L\^e cycles and L\^e numbers \cite{MasseyLeCycles}. Teissier's polar varieties and polar multiplicities relate this deformation-theoretic information to local Chern-theoretic invariants \cite{LeTeissier81}. Related connections between the Euler obstruction and indices of vector fields were established by Brasselet, L\^e and Seade \cite{BLS}.  The behaviour of polar invariants in families is governed by Gaffney's multiplicity-polar theorem \cite{Gaffney92,Gaffney93}. Corollary~\ref{coro:gradient-introduction} supplies a complementary lower bound in terms of the embedded contribution associated with $W$. In the projective complete-intersection setting considered above, the residual term $\nu(\fol,W,\varphi_0)$ is explicitly determined by the degrees and by the order of vanishing along the exceptional divisor.

\section{Preliminaries}

\subsection{$\kk$-equivalence}
We use Mather's definition of $\mathcal K$-determinacy for map germs \cite{MT1,MT2}. Let $\mathcal O_{n,n}$ denote the space of analytic map germs from $\mathbb C^n$ to $\mathbb C^n$, and let $\mathcal K$ be the group of germs of biholomorphisms
\[
H:(\mathbb{C}^n\times \mathbb{C}^n,(0,0))\longrightarrow (\mathbb{C}^n\times \mathbb{C}^n,(0,0))
\]
for which there exists a germ of biholomorphism $\varphi:(\mathbb{C}^n,0)\to (\mathbb{C}^n,0)$ such that
\begin{itemize}
\item[(i)] $H(x,0)=(\varphi(x),0)$ for all $x$;
\item[(ii)] if $f,g\in\mathcal{O}_{n,n}$ are such that $H$ sends the graph of $f$ to the graph of $g$ over $\varphi$, namely
\[
H(x,f(x))=(\varphi(x),g(\varphi(x))),
\]
then $f$ and $g$ are said to be $\mathcal{K}$-equivalent.
\end{itemize}
Equivalently, $f$ and $g$ are $\mathcal{K}$-equivalent if there exist $H\in\mathcal{K}$ and $\varphi$ as above such that
\[
H\circ (1,f)\circ \varphi^{-1}=(1,g)
\]
as germs $(\mathbb{C}^n,0)\to(\mathbb{C}^n\times \mathbb{C}^n,(0,0))$.
With this notation, a germ $f\in\mathcal{O}_{n,n}$ is called \textbf{finitely $\mathcal{K}$-determined} if there exists an integer $\ell\geq 0$ such that for every germ $g\in\mathcal{O}_{n,n}$ with $\mathbf{J}_0^{\ell}g=\mathbf{J}_0^{\ell}f$, the germs $f$ and $g$ are $\mathcal{K}$-equivalent.
In the equidimensional case, finite $\mathcal K$-determinacy is characterised as follows.

\begin{lemma}
A germ $f\in\mathcal{O}_{n,n}$ is finitely $\mathcal{K}$-determined if and only if $f^{-1}(0)$ has an isolated singularity at the origin.
\end{lemma}

When the dimension of the source exceeds that of the target, we use transverse equidimensional slices.
Let $f=(f_1,\ldots,f_d):(\mathbb{C}^n,0)\to(\mathbb{C}^d,0)$ be a holomorphic map germ with $1\leq d\leq n$, and assume that
$
\W:=Z(f_1,\ldots,f_d)
$
is a smooth complete intersection of codimension $d$.
Then, for each $z\in W$, some $d\times d$ minor of the Jacobian matrix $Jf$ is invertible.
After a linear change of coordinates, fix a decomposition $\mathbb{C}^n=\mathbb{C}^d\oplus\mathbb{C}^{n-d}$ and write $z=(x,y)$, where $x\in\mathbb{C}^d$ and $y\in\mathbb{C}^{n-d}$.
For each fixed $y\in\mathbb{C}^{n-d}$ we consider the induced equidimensional germ
\[
f_y:(\mathbb{C}^d,0)\longrightarrow (\mathbb{C}^d,0),
\qquad
f_y(x):=f(x,y)=\bigl(f_1(x,y),\ldots,f_d(x,y)\bigr).
\]
We use the following parameter-dependent analogue of $\mathcal K$-equivalence.
Let $\mathcal{K}^{(d)}$ be the group of families of biholomorphism germs
\[
H_y^{(d)}:(\mathbb{C}^d\times \mathbb{C}^d,(0,0))\longrightarrow (\mathbb{C}^d\times \mathbb{C}^d,(0,0)),
\]
depending holomorphically on the parameters $y=(y_1,\ldots,y_{n-d})$, such that for generic parameters $y\in\mathbb{C}^{n-d}$ there exists a biholomorphism
$\varphi_y:(\mathbb{C}^d,0)\to (\mathbb{C}^d,0)$ with the properties
\begin{itemize}
\item[(i)] $H_y^{(d)}(x,0)=(\varphi_y(x),0)$;
\item[(ii)] $H_y^{(d)}(x,f_y(x))=(\varphi_y(x), g_y(\varphi_y(x)))$,
\end{itemize}
where $g_y$ is defined analogously from a second germ $g$.

\begin{definition}
Two holomorphic map germs $f,g:(\mathbb{C}^n,0)\to(\mathbb{C}^d,0)$ are \textbf{$\mathcal{K}^{(d)}$-equivalent} if for generic $y\in\mathbb{C}^{n-d}$ there exist
$H_y^{(d)}\in \mathcal{K}^{(d)}$ and an associated biholomorphism $\varphi_y:(\mathbb{C}^d,0)\to(\mathbb{C}^d,0)$ such that
\begin{itemize}
\item[(i)] $H_y^{(d)}(x,0)=(\varphi_y(x),0)$;
\item[(ii)] $H_y^{(d)}(x,f_y(x))=(\varphi_y(x), g_y(\varphi_y(x)))$,
\end{itemize}
for all $x$ in a neighbourhood of $0\in\mathbb{C}^d$.
\end{definition}

\begin{theorem}[Analytic Artin approximation, cf.\ \cite{ArtinApprox2018}]\label{thm:ArtinApprox}
Let $f(x,y)$ be a vector of convergent power series in two sets of variables $x$ and $y$, namely
$f\in \mathbb{C}\{x,y\}^{m}$.
Assume that a formal power series vector $\widehat y(x)\in \mathbb{C}[[x]]^{n}$ satisfies $\widehat y(0)=0$ and $
f\bigl(x,\widehat y(x)\bigr)=0.
$
Then, for every $c\in\mathbb{N}$, there exists a convergent power series solution
$\widetilde y(x)\in \mathbb{C}\{x\}^{n}$ such that
\[
f\bigl(x,\widetilde y(x)\bigr)=0
\quad\text{and}\quad
\widetilde y(x)\equiv \widehat y(x)\pmod{(x)^c}.
\]
\end{theorem}

Theorem~\ref{thm:ArtinApprox} replaces formal solutions by convergent solutions to any prescribed finite order. The following definition is the corresponding parameter-dependent form of Mather's $\mathcal K$-equivalence \cite{MT1,MT2}.

Two holomorphic map germs $f,g\in\mathcal O_{n,n}$ have the same $k$-\emph{jet} at $0$ if their partial derivatives at $0$ agree through order $k$, or equivalently if their Taylor series agree through total degree $k$. The $k$-jet of $f$ at $0$ is denoted by $\mathbf J_0^kf$.

\begin{definition}\label{def:Kd-determinacy}
A germ $f:(\mathbb{C}^n,0)\to(\mathbb{C}^d,0)$ is \emph{finitely $\mathcal{K}^{(d)}$-determined} if there exists an integer $\ell\geq 0$
such that, for every germ $g:(\mathbb{C}^n,0)\to(\mathbb{C}^d,0)$,
\[
\mathbf{J}_0^{\ell}g=\mathbf{J}_0^{\ell}f
\quad\Longrightarrow\quad
g \text{ is } \mathcal{K}^{(d)}\text{-equivalent to } f.
\]
The \emph{degree of $\mathcal{K}^{(d)}$-determinacy} of $f$ is the integer
$
\ell_0:=\min\bigl\{\ell\in\mathbb{N}\ \big|\ f\text{ is } \ell\text{-}\mathcal{K}^{(d)}\text{-determined}\bigr\}.
$
\end{definition}

In local coordinates $\mathbb C^n=\mathbb C^d\oplus\mathbb C^{n-d}$, with $z=(x,y)$, each fixed parameter $y$ determines the equidimensional germ
$
f_y:(\mathbb{C}^d,0)\to(\mathbb{C}^d,0),
$ $
f_y(x):=f(x,y),
$
and, by construction, $f_y(0)=0$ for $y$ near $0$.
For $y$ outside a proper analytic subset, classical Mather theory applies to the slice $f_y$.

\medskip

\begin{lemma}\label{subvariedad}
Let $W=Z(f_1,\ldots,f_d)\subset(\mathbb{C}^n,0)$ be a germ of a smooth complete intersection, where each $f_i$ is reduced, and set
$f=(f_1,\ldots,f_d):(\mathbb{C}^n,0)\to(\mathbb{C}^d,0)$.
Then there exist local coordinates $(x,y)$ on $\mathbb{C}^n=\mathbb{C}^d\oplus \mathbb{C}^{n-d}$ such that,
for generic $y\in\mathbb{C}^{n-d}$, the slice
\[
f_y:(\mathbb{C}^d,0)\to(\mathbb{C}^d,0),
\quad
f_y(x)=(f_1(x,y),\ldots,f_d(x,y)),
\]
is finitely $\mathcal{K}$-determined (in the classical, equidimensional sense).
\end{lemma}

\begin{proof}
Since $W$ is smooth of codimension $d$ at $0$, the Jacobian matrix of $f$ has rank $d$ at $0$.
After reordering coordinates if necessary, we may assume that the $d\times d$ minor
\begin{equation}\label{eq:Df}
Df(0)=\left[\frac{\partial f_i}{\partial z_j}(0)\right]_{1\leq i,j\leq d}
\end{equation}
is invertible.
Consider the holomorphic map germ
\[
\psi:(\mathbb{C}^n,0)\to(\mathbb{C}^n,0),
\qquad
\psi(z)=\bigl(f_1(z),\ldots,f_d(z),z_{d+1},\ldots,z_n\bigr).
\]
By \eqref{eq:Df} and the holomorphic inverse function theorem, $\psi$ is a local biholomorphism near $0$.
In the new coordinates $z'=\psi(z)=(x,y)$, the map $f$ becomes the projection to the first $d$ coordinates:
\[
f\circ \psi^{-1}(x,y)=(x_1,\ldots,x_d).
\]
For every $y$ near $0$, the slice $(f\circ\psi^{-1})_y$ is therefore the identity germ on $(\mathbb C^d,0)$. It has an isolated zero at $x=0$ and is finitely $\mathcal K$-determined. Transport by $\psi$ proves the assertion for $f_y$; in fact it holds for every sufficiently small $y$.
\end{proof}

Each component $f_i(x,y)$ has a convergent expansion in the normal variables $x$, with coefficients holomorphic in $y$:
\begin{equation}\label{eq:series-x}
f_i(x,y)=\sum_{a\in\mathbb{N}^d} x^a\, f_{i,a}(y),
\qquad
f_{i,a}(y)\in \mathbb{C}\{y\},
\end{equation}
where $a=(a_1,\ldots,a_d)$, $x^a=x_1^{a_1}\cdots x_d^{a_d}$, and $|a|=a_1+\cdots+a_d$.
Writing $f_{i,a}(y)=\sum_{b\in\mathbb{N}^{n-d}}\alpha_{i,a,b}y^b$ with $y^b=y_1^{b_1}\cdots y_{n-d}^{b_{n-d}}$, we obtain
\begin{equation}\label{eq:double-series}
f_i(x,y)=\sum_{a\in\mathbb{N}^d}\ \sum_{b\in\mathbb{N}^{n-d}} \alpha_{i,a,b}\, x^a y^b.
\end{equation}

\begin{example}
Let $f:(\mathbb{C}^n,0)\to(\mathbb{C},0)$ be a holomorphic germ whose critical set is smooth of codimension one.
After a linear change of coordinates we may assume that the critical set is the hyperplane
$
\Sigma=\{z\in\mathbb{C}^n\mid z_1=0\}.
$
Writing $x=z_1$ and $y=(z_2,\ldots,z_n)$, we may express $f$ as
$
f(x,y)=\sum_{j=m}^{\infty} a_j(y)\,x^j,\ 
a_m(0)\neq 0,
$
so $a_m(y)$ is a unit in $\mathbb{C}\{y\}$ after shrinking the neighbourhood of $0$.
Choose a holomorphic function $b_1(y)$ with $b_1(y)^m=a_m(y)$ (after restricting to a simply connected neighbourhood if needed),
and define a biholomorphism germ $\varphi_y:(\mathbb{C},0)\to(\mathbb{C},0)$ of the form
$
\varphi_y(x)=b_1(y)\,x+\sum_{j\geq 2} b_j(y)\,x^j
$
such that $\bigl(\varphi_y(x)\bigr)^m=f(x,y)$ (one may construct $\varphi_y$ recursively by comparing coefficients).
Since $b_1(y)\neq0$ near $0$, the germ $\varphi_y$ is biholomorphic. Setting $g(x,y)=x^m$ gives a $\mathcal K^{(1)}$-equivalence between $f$ and $g$: for $H_y(x,w)=(\varphi_y(x),w)$,
\[
H_y\bigl(x,f(x,y)\bigr)=\bigl(\varphi_y(x),f(x,y)\bigr)=\bigl(\varphi_y(x),(\varphi_y(x))^m\bigr)
=\bigl(\varphi_y(x),g(\varphi_y(x),y)\bigr).
\]
\end{example}

\section{Holomorphic vector fields with non-isolated singularities}
First fix the neighbourhoods in which the local indices are to be counted. Let $W$ be a codimension-$d$ complete intersection in a complex manifold $M$.
Fix $\epsilon>0$. For each $z\in W$, let $D_\epsilon(z)$ denote the disc of radius $\epsilon$ centred at $z$ in a local holomorphic chart.
Define the $\epsilon$-neighbourhood of $W$ by
\[
W_\epsilon:=\bigcup_{z\in \W} D_\epsilon(z).
\]
An open set $V_\epsilon$ is called a precompact $\epsilon$-neighbourhood of $W$ if $V_\epsilon=U\cap W_\epsilon$ for some relatively compact open subset $U\subset M$.

Let $\F$ be a one-dimensional holomorphic foliation on $\C^n$ whose singular set consists of a smooth complete intersection manifold
$W\subset \C^n$ and, possibly, finitely many isolated points.
Assume that $\W=Z(f_1,\ldots,f_d)$, where each $f_i:\C^n\to\C$ is holomorphic and $W$ is smooth of pure codimension $d\geq2$.
Unless stated otherwise, we work under the standing assumption that
\[
\sing(\F)=\W\cup\{p_1,\ldots,p_r\},
\]
for some $r\in \Z_{\geq 0}$.
On each chart $(U_\alpha,\varphi_\alpha)$, the foliation is generated by a holomorphic vector field
\[
X_{\alpha}=X|_{U_{\alpha}}=\sum_{i=1}^n X_i \frac{\partial}{\partial z_i}=(X_1,\ldots,X_n).
\]

Let $f:\C^n\to \C^d$ be defined by $f(z)=(f_1(z),\ldots,f_d(z))$.
By the holomorphic submersion theorem (after possibly shrinking the chart),
the map
\[
u=\psi(z)=\bigl(f_1(z),\ldots,f_d(z),z_{d+1},\ldots,z_n\bigr)=(\omega_1,\omega_2)\in \C^d\oplus \C^{n-d}
\]
induces a change of coordinates on $\C^n$ such that $\psi(W)=W_0$ is given by $u_1=\cdots=u_d=0$, equivalently $\omega_1=0$.
Let
\[
Y=\psi_*X=\sum_{i=1}^n Y_i \frac{\partial}{\partial w_i}.
\]
In these coordinates, each component has an expansion
\begin{equation}\label{multipli.pu}
Y_i(u):=\sum_{|a|=m_i} u_1^{a_1}u_2^{a_2}\cdots u_d^{a_d}\,g_{i,a}(u),
\end{equation}
where, for each $i$, at least one coefficient $g_{i,a}\in\oo_n$ does not lie in the ideal of $W_0$ on a suitable precompact neighbourhood.
Write $m_i:={\rm{mult}}_W(Y_i)$.
As in \cite{MAGR}, we define the multiplicity of $\F$ along $W$ by
\[
m_W(\F):=\min\{m_1,\ldots,m_n\}.
\]

Because $\W_0$ is defined by $u_1=\cdots=u_d=0$ (equivalently $\omega_1=0$), each local section $Y_i$ of
\[
Y=\psi_*X=\sum_{i=1}^n Y_i \frac{\partial}{\partial w_i}
\]
can be written as
\[
Y_i(w)=\sum_{|a|=m_i}^{\infty} u_1^{a_1}\cdots u_d^{a_d}\,Y_{i,a}(u_{d+1},\ldots,u_n)
=\sum_{|a|=m_i}^{\infty} \omega_1^{a}\,Y_{i,a}(\omega_2),
\]
where $a=(a_1,\ldots,a_d)\in \N^d$ and for each $i$ at least one coefficient $Y_{i,a}$ does not vanish along $\omega_1=0$.

\begin{theorem}\label{thm:polynomial-approximation}
Let $\F$ be a one-dimensional holomorphic foliation on $\C^n$ with
$\W=Z(f_1,\ldots,f_d)\subset\sing(\F)$ a smooth complete intersection on $\C^n$ and each $f_i\in \oo_{\C^n}$.
Then there are polynomial approximations $\{\F_\kappa\}$ of $\F$ and analytic subsets $W_\kappa\subset\sing(\F_\kappa)$ such that:
\begin{enumerate}
\item[(a)] $\W_{\kappa}$ is a smooth algebraic local complete intersection for all $\kappa$;
\item[(b)] $\displaystyle\lim_{\kappa\to \infty}\W_{\kappa}=\W$ on each precompact $\epsilon$-neighbourhood of $W$.
\end{enumerate}
\end{theorem}

\begin{proof}
Let $X:=\{X_\alpha,U_\alpha\}$ be a vector field defining $\F$ in local coordinates $\{U_\alpha,\varphi_\alpha\}$, and write
\begin{equation}\label{reference_field}
X(z):=\sum_{i=1}^n X_i(z)\frac{\partial}{\partial z_i}=(X_1(z),\ldots,X_n(z)).
\end{equation}
Let $f_i\in \oo_{\C^n}$ be germs of holomorphic functions such that
$\W:=Z(f_1,\ldots,f_d)$ is a smooth complete intersection subvariety.
After reordering the coordinates if necessary, we may assume that a $d\times d$ minor of the Jacobian matrix $Jf$ of $f$
is invertible on a neighbourhood $U$ of some point of $W$.
Hence the map $\varphi:\C^n\to\C^n$ given by
\[
z\longmapsto \varphi(z)=(f_1(z),\ldots,f_d(z),z_{d+1},\ldots,z_n)=w
\]
is a local biholomorphism.
The push-forward $Y=\varphi_*X$ is a holomorphic vector field such that
$\W_0:=\varphi(W)=Z(w_1,\ldots,w_d)\subset\sing(Y)$ is the common zero locus of the coordinate functions $w_i$ for $i=1,\ldots,d$, and
\begin{equation}\label{eq:push-forward-field}
Y(w):=\varphi_*X(u)=\sum_{i=1}^n Y_i(w)\frac{\partial}{\partial w_i}.
\end{equation}

Fix a relatively compact open $\epsilon/2$-neighbourhood $V\subset\C^n$ such that the set of embedded points $\mathcal{A}_{\W}$
is contained in $V$.
After shrinking $V$ if necessary, we may assume that $\overline V$ is contained in a polydisc centred at the origin on which all coefficient functions $X_i$
in \eqref{reference_field} admit convergent power series expansions.
Let $X_\kappa$ be the holomorphic vector field obtained by truncating the power series of each coefficient $X_i$ to total degree $\leq \kappa$.
Then $X_\kappa$ is polynomial and $X_\kappa\to X$ uniformly on $\overline V$ as $\kappa\to\infty$.
For a fixed sufficiently small $\epsilon>0$, choose $\kappa$ so large that $X_\kappa$ is $\epsilon/2$-close to $X$ on $\overline V$, and retain this value of $\kappa$.

Let $Y=\varphi_*X$ be the vector field inducing the push-forward one-dimensional foliation $\mathcal{G}$, where
$\varphi(z)=(f_1(z),\ldots,f_d(z),z_{d+1},\ldots,z_n)$.
Keeping this notation, denote by $\mathcal{G}_\kappa$ the foliation induced by the polynomial vector field $Y_\kappa$
associated with $X_\kappa$ via the corresponding truncation $\varphi_\kappa$.
Since $\W_0=\varphi(\W)$ is a smooth complete intersection in $\C^n$, its projective closure, obtained by homogenisation, is a smooth complete intersection in $\mathbb P^n$.

To construct the required polynomial deformation, let $Y_\kappa$ be as in \eqref{eq:push-forward-field} and set
\begin{equation}
Y_{\kappa,t}=Y_\kappa+t\,\widetilde{Y},
\end{equation}
where
\[
\widetilde{Y}=Q_1\frac{\partial}{\partial w_1}+\cdots+Q_n\frac{\partial}{\partial w_n},
\]
and
\[
Q_j(w)=
\begin{cases}
\displaystyle\sum_{|I|=q_j} a_{I,j}\, w_1^{i_1}\cdots w_d^{i_d} & \text{if } 1\leq j\leq d,\\[6pt]
\displaystyle\sum_{|I|=q_j} R_{I,j}(w)\, w_1^{i_1}\cdots w_d^{i_d} & \text{if } d+1\leq j\leq n,
\end{cases}
\]
where $a_{I,j}\in\C$ and each $R_{I,j}$ is an affine linear function.
Impose
\[
q_1=\cdots=q_d=q_{d+1}+1=\cdots=q_n+1=\ell+1,
\]
and choose the coefficients so that $\deg(Q_j)\leq \deg(Y_\kappa)$ and the hyperplane at infinity $H^\infty$ is invariant under $Y_{\kappa,t}$.
After an arbitrarily small perturbation of the coefficients of $\widetilde Y$, we may assume that $Y_{\kappa,t}$ is special along $\W_0$ and that $\mathcal G_{\kappa,t}$ has no embedded points associated with $\W_0$ for $t\neq0$. The deformation then satisfies:
\begin{enumerate}
\item[(i)] $\mathcal{G}_{\kappa,0}=\mathcal{G}_\kappa$;
\item[(ii)] $\deg(\mathcal{G}_{\kappa,t})=\deg(\mathcal{G}_\kappa)$;
\item[(iii)] If $\ell=0$, then $\sing(\mathcal{G}_{\kappa,t})=\{p_1^t,\ldots,p_{s_t}^t\}$ and $W_0$ is $\mathcal{G}_{\kappa,t}$-invariant for every $t\in D_\epsilon\setminus\{0\}$;
\item[(iv)] If $\ell>0$, then $\sing(\mathcal{G}_{\kappa,t})=W_0\cup\{p_1^t,\ldots,p_{s_t}^t\}$ is a disjoint union for every $t\in D_\epsilon\setminus\{0\}$, and $\mathcal{G}_{\kappa,t}$ is special along $W_0$ with $\ell=m_E(\mathcal{G}_{\kappa,t},W_0)$.
\end{enumerate}

The polynomial vector field $Y_{\kappa,t}$ defines a holomorphic foliation on all of $\C^n$. Properties (i) and (ii) follow from the definition; properties (iii) and (iv) follow from a suitable choice of the coefficients $a_{I,j}$ and $R_{I,j}$.

After reducing $\epsilon$ if necessary, choose a polynomial vector field
\[
\alpha=\sum_{i=1}^n \alpha_i(w)\frac{\partial}{\partial w_i},
\qquad
\alpha_i\in \C[w_1,\ldots,w_n],
\qquad
\deg(\alpha)\leq \deg(Y_\kappa),
\]
such that the deformation
\begin{equation}\label{eq:principal-deformation}
\widetilde{Y}_t
=
Y_{\kappa,t}+t\sum_{i=1}^n \alpha_i \frac{\partial}{\partial w_i}
\end{equation}
has only isolated singularities for every $t\ne 0$ on the chosen relatively compact $\epsilon/2$-neighbourhood.
Equivalently, the foliations induced by $\widetilde{Y}_t$ have only isolated singularities on that neighbourhood whenever $t\ne 0$.
For every sufficiently small $t\neq0$, a neighbourhood of each embedded point of $Y$ contains singular points of $\widetilde Y_t$, and their total multiplicity equals the number of embedded points associated with $W_0$. Hence the singularities of $\widetilde Y_t$ converging to $W_0$ form a finite set of constant total multiplicity for $0<|t|<\epsilon/2$.
To prove minimality, let $\{Z_t\}_{t\in D'_\epsilon}$ be any other generic holomorphic deformation of $Y$ whose singularities in the same relatively compact $\epsilon/2$-neighbourhood $V$ are isolated.
Fix a norm $\|\cdot\|$ on $\oo_{\C^n}(V)$.
Since $\|Z_t-Y\|\leq \epsilon/2$ and $\|\widetilde{Y}_t-Y\|\leq \epsilon/2$, it follows that
$\|\widetilde{Y}_t-Z_t\|<\epsilon$ on $V$ for $\epsilon$ sufficiently small.
Hence the number of points of $\sing(Z_t)$, counted with multiplicity, is at least the number of points in $\mathcal A_{\W}$. By continuity, the corresponding isolated singularities converge to $\W_0$.
\end{proof}

Fix a relatively compact $\epsilon$-neighbourhood $V$ of $\W_0$ containing all embedded closed points associated with $\W_0$.
Then
\begin{equation}\label{eq:NW0-muY}
N_{\W_0}
=\mu(Y,\W_0)
=\lim_{t\to0}\sum_{p_i^t\in \Omega_t}\mu(Y_t,p_i^t).
\end{equation}
Since $Y=X_\kappa$, it follows that for every $\kappa$ one has
\begin{equation}\label{eq:Nkappa-mu}
N^\kappa_{\W_0}
=\mu(X_\kappa,\W_0)
=\lim_{t\to 0}\sum_{p_i^t\in A_{W_0}} \mu\bigl((X_\kappa)_t,p_i^t\bigr)
\geq N_{\W_0}.
\end{equation}

\begin{definition}\label{def:totally-simple}
A foliation $\fol$ is said to be \emph{totally simple} along a positive-dimensional component
$\W\subset \sing(\fol)$ if, for every point $q\in\W$, there exists a local holomorphic vector field $X$ representing $\fol$ and a $d\times d$ minor of the Jacobian matrix $JX(q)$ all of whose eigenvalues are non-zero.
\end{definition}

\begin{theorem}\label{thm:totally-simple-perturbation}
If $\fol$ is totally simple along $\W$, then there exists a holomorphic perturbation $\{\fol_t\}$ of $\fol$
such that $\fol_0=\fol$, $\lim_{t\to0}\fol_t=\fol$, and
$
\Omega_t:=\{\,p_t\in \sing(\fol_t)\mid p_t\to \W \text{ as } t\to0\,\}=\varnothing.
$
Consequently,
\begin{equation}\label{eq:totally-simple-mu-zero}
\lim_{t\to 0}\mu(\fol_t,\W)=0.
\end{equation}
\end{theorem}

\begin{proof}
Let $V$ be a precompact $\epsilon$-neighbourhood of $\W$.
Let $f=(f_1,\ldots,f_d)$ be as above, so that $\W=Z(f_1,\ldots,f_d)$, and let
$
\varphi:=(f_1,\ldots,f_d,z_{d+1},\ldots,z_n)
$
be a local change of coordinates on $\C^n$.
Let $Y=\varphi_*X$ be the push-forward of a local holomorphic vector field $X$ representing $\fol$, so that
\[
Y(w)=\sum_{i=1}^n Y_i(w)\,\frac{\partial}{\partial w_i}.
\]
Denote by $\W_0:=\varphi(\W)$ the image of $\W$; then $\W_0$ is locally given by $\{w_1=\cdots=w_d=0\}$.
Each component $Y_i$ has an expansion
\begin{equation}\label{eq:Yi-expansion}
Y_i(w)=\sum_{|\alpha|\geq m_i} w_1^{\alpha_1}\cdots w_d^{\alpha_d}\,Y_{i,\alpha}(w),
\qquad
\alpha=(\alpha_1,\ldots,\alpha_d),\ \ |\alpha|=\alpha_1+\cdots+\alpha_d,
\end{equation}
for suitable holomorphic functions $Y_{i,\alpha}$.
By total simplicity, the Jacobian matrix $JY(w)$ has, along $\W_0\cap \widetilde V$ (where $\widetilde V:=\varphi(V)$),
a $d\times d$ submatrix whose determinant is nowhere vanishing. After reordering coordinates if necessary,
we may assume that this is the submatrix corresponding to the first $d$ components $(Y_1,\ldots,Y_d)$
and the variables $(w_1,\ldots,w_d)$. Hence the linear part of $(Y_1,\ldots,Y_d)$ in the variables $(w_1,\ldots,w_d)$ is invertible along $\W_0\cap\widetilde V$, and $m_{W_0}(Y)=1$.
Equivalently, on $\widetilde V$ we may write
\begin{equation}\label{eq:Yi-linear}
Y_i(w)=\sum_{j=1}^d a_{i,j}(w)\,w_j,
\qquad i=1,\ldots,n,
\end{equation}
for suitable holomorphic functions $a_{i,j}$, with the $d\times d$ matrix $(a_{i,j})_{1\leq i,j\leq d}$
invertible along $\W_0\cap\widetilde V$.
Choose constants $\epsilon_{d+1},\ldots,\epsilon_n\in\C^*$ and consider the holomorphic deformation
$\{Y_t\}_{t\in D_\epsilon}$ defined by
\begin{equation}\label{eq:Yt-translation}
Y_t
=
Y+t\Bigl(0,\ldots,0,\epsilon_{d+1},\ldots,\epsilon_n\Bigr).
\end{equation}
For every sufficiently small $t\neq0$, the vector field $Y_t$ has no zeros on $\widetilde V$. Suppose, to the contrary, that $Y_t(w)=0$. Its first $d$ components then vanish:
$Y_1(w)=\cdots=Y_d(w)=0$. By \eqref{eq:Yi-linear} and the invertibility of $(a_{i,j})_{1\leq i,j\leq d}$ on
$\widetilde V\cap\W_0$, this forces $w_1=\cdots=w_d=0$, hence $w\in\W_0$.
But on $\W_0$ the last $n-d$ components of $Y_t$ equal
$
Y_{d+1}(w)+t\epsilon_{d+1},\ \ldots,\ Y_n(w)+t\epsilon_n,
$
and since $Y|_{\W_0}\equiv 0$ while $\epsilon_{d+1},\ldots,\epsilon_n\ne 0$, none of these can vanish for $t\ne0$.
This is impossible. Set $X_t:=\varphi^*Y_t$. Then $X_t\to X$ as $t\to0$, and $X_t$ has no zeros on $V$ for every sufficiently small $t\neq0$. Thus $\Omega_t=\varnothing$ and $\mu(\fol_t,\W)=0$, which proves \eqref{eq:totally-simple-mu-zero}.
\end{proof}

\subsection{Gradient fields and morsifications}\label{subsec:gradient-fields}

Let
$
f:(\mathbb C^n,0)\longrightarrow(\mathbb C,0)
$
be a holomorphic germ and set $X=\nabla f$. The singular locus of $X$ is the critical locus of $f$, and the Poincar\'e--Hopf index of the gradient at an isolated critical point is the corresponding Milnor number. The local theorem therefore applies directly to deformations of $f$ whose nearby critical loci are finite.

\begin{proof}[Proof of Corollary~\ref{coro:gradient-introduction}]
Set $X_t=\nabla f_t$. Then
$
\sing(X_t)=\operatorname{Crit}(f_t).
$
At every isolated critical point $p$ of $f_t$, one has
\[
\mathcal I_p(X_t)
=
\dim_{\mathbb C}
\frac{\mathcal O_{\mathbb C^n,p}}
{\left(
\frac{\partial f_t}{\partial z_1},\ldots,
\frac{\partial f_t}{\partial z_n}
\right)}
=
\mu(f_t,p);
\]
see \cite{Milnor}. The first assertion follows from Theorem~\ref{theorem2}. If $f_t$ is a morsification, every critical point of $f_t$ is non-degenerate and hence has Milnor number one. The second assertion follows.
\end{proof}

\begin{remark}
The deformation of a positive-dimensional critical locus may also be described by L\^e cycles, L\^e numbers and polar multiplicities \cite{LeTeissier81,MasseyLeCycles}. Gaffney's multiplicity-polar theorem controls the behaviour of these invariants in families \cite{Gaffney92,Gaffney93}. Corollary~\ref{coro:gradient-introduction} isolates the lower bound furnished by the embedded contribution associated with $W$.
\end{remark}

\section{Examples}\label{sec:examples}

The following examples exhibit the phenomena specific to non-isolated singular sets. They show how isolated singularities may separate from a positive-dimensional component and return to it, how the total index may be redistributed between a fixed affine neighbourhood and the hyperplane at infinity, and how the contribution along the component may vanish in the totally simple case. They also show that the lower bounds proved above are sharp.

In the first example, $\W=\{z_1=z_2=0\}\subset\C^3$. Two perturbations are considered. The resulting counts depend on the common roots of $\alpha_3$ and the relevant characteristic polynomials, and determine the possible values of $\mu(X_t,\W)$ and the contribution $N(\fol,A_{\W_0})$.

\begin{example}
Consider the following family of polynomial vector fields on $\C^3$:
\[
X=\sum_{i=0}^{m}a_iz_1^{k-i}z_2^i\frac{\partial}{\partial z_1}
+\sum_{i=0}^{m}b_iz_1^{k-i}z_2^i\frac{\partial}{\partial z_2}
+\sum_{i=0}^{m-1}c_i(z)z_1^{m-i-1}z_2^i\frac{\partial}{\partial z_3},
\]
where $m\geq 2$, $c_i(z)=\alpha_{i,0}+\sum_{j=1}^{3}\alpha_{i,j}z_j$ is an affine linear function for
$i=0,\ldots,m-1$, and $a_i,b_i\in\C$. Set
\[
\alpha_j(\lambda)=\sum_{i=0}^{m-1}\alpha_{i,j}\lambda^i \qquad (j=0,1,2,3).
\]
Assume that
$
a(\lambda)=\sum_{i=0}^{m}a_i\lambda^i
 $ and $
b(\lambda)=\sum_{i=0}^{m}b_i\lambda^i
$
have no common roots. Thus, $\deg(a)=\deg(b)=\deg(\alpha_j)+1=m$.
Assume further that
$
\sing(X)=\W=\{z\in\C^3\mid z_1=z_2=0\}.
$

Let $\fol$ be the one-dimensional holomorphic foliation on $\mathbb{P}^3$ induced, on the affine chart $\C^3$,
by the vector field $X$. Thus, the singular set of $\fol$ consists of
$
\W_0=\{[x]\in\mathbb{P}^3\mid x_0=x_1=0\},
$
(with $z_i=x_{i-1}/x_3$ for $i=1,2,3$) together with $m+1$ isolated closed points, counted with multiplicity.
Namely,
\[
p_j=[u_{1j}:\beta_j u_{1j}:1:0]\in\sing(\fol),
\]
where $\beta_j$ is a root of the equation $b(\lambda)-\lambda a(\lambda)=0$ and
\[
u_{1j}=\frac{a(\beta_j)-\alpha_3(\beta_j)}{\alpha_1(\beta_j)+\beta_j\alpha_2(\beta_j)}.
\]
Thus, if $u_{1j}=0$ for some $j$, then there is at least one embedded point of $\sing(\fol)$ associated with $\W_0$.
Therefore, the number of embedded singular points of $\sing(\fol)$ associated with $\W_0$ satisfies
$N(\fol,A_{\W_0})\leq m-1$, since $a(\lambda)$ and $\alpha_3(\lambda)$ have at most $m-1$ common roots.
If $\alpha_3\equiv0$, suitable coefficients give $N(\fol,A_{\W_0})=m$.
To determine the Milnor number $\mu(X_t,\W)$, consider the holomorphic deformation $X_t$ of $X$ on $\C^3$ given by
\begin{equation}\label{genpert}
X_t=X-t\bigg(\epsilon_1\frac{\partial}{\partial z_1}
+\epsilon_2\frac{\partial}{\partial z_2}
+\epsilon_3\frac{\partial}{\partial z_3}\bigg),
\end{equation}
where each $\epsilon_i\ne 0$ is a generic complex number. Thus,
\[
X_t^{-1}(0)=\{p_{ij}^t=(z_{1ij}^t,\lambda_j z_{1ij}^t,z_{3ij}^t)\in \C^3,\ 1\leq i,j\leq m\},
\]
where $\lambda_j$ is a root of the equation $\epsilon_1b(\lambda)-\epsilon_2a(\lambda)=0$,
$z_{1ij}^t$ is an $m$-th root of $t\epsilon_1/a(\lambda_j)$, and
\[
z_{3ij}^t=-\frac{\alpha_0(\lambda_j)}{\alpha_3(\lambda_j)}
+z_{1ij}^t\frac{\epsilon_3a(\lambda_j)-\epsilon_1(\alpha_1(\lambda_j)+\lambda_j\alpha_2(\lambda_j))}
{\epsilon_1\alpha_3(\lambda_j)},
\]
where we necessarily assume $\alpha_3(\lambda_j)\ne 0$. Under this condition,
\[
\lim_{t\to 0}z_{1ij}^t=0,
\qquad
\lim_{t\to 0}z_{3ij}^t=-\frac{\alpha_0(\lambda_j)}{\alpha_3(\lambda_j)}.
\]
Hence,
\begin{equation}\label{pertX}
\mu(X_t,\W)= \lim_{t\to0}\sum_{p_{ij}^t\in \mathcal{A}_{\W}}\mu(X_t,p_{ij}^t)=m^2
\end{equation}
for a generic perturbation $X_t$ as in \eqref{genpert}.
However, if there is exactly one $t_j$ such that \[\epsilon_1 b(t_j)-\epsilon_2a(t_j)=0,\] where $t_j$ is also a root of $\alpha_3$,
then the limit in \eqref{pertX} equals $m^2-m$. In the extreme case, when the equations
$\epsilon_1 b(\lambda)-\epsilon_2a(\lambda)=0$ and $\alpha_3(\lambda)=0$ have $m-1$ common roots,
the minimum value of $\mu(X_t,\W)$ is $m^2-(m-1)m=m$, provided that $\alpha_3\not\equiv 0$.
If $\alpha_3\equiv 0$, suitable choices give the minimum value $\mu(X_t,\W)=0$.
When $\alpha_3\equiv0$, consider instead the holomorphic perturbations
\begin{equation}\label{genpert2}
Y_t=X-t\bigg(\epsilon_1\frac{\partial}{\partial z_1}
+\epsilon_2\frac{\partial}{\partial z_2}
+(z_3^m+\epsilon_3)\frac{\partial}{\partial z_3}\bigg).
\end{equation}
In this case, for $t\ne 0$ the singular set of $Y_t$ consists generically of $m^3$ isolated points, counted with multiplicity.
Indeed,
\[
\sing(Y_t)=\{ q^t_{ijk}=(z^t_{1ij},\lambda_j z^t_{1ij},z^t_{3ijk})\in\C^3\},
\]
where $\lambda_j$ is a root of $\epsilon_1b(\lambda)-\epsilon_2a(\lambda)=0$,
$z_{1ij}^t$ is an $m$-th root of $t\epsilon_1/a(\lambda_j)$,
and $z^t_{3ijk}$ is a root (in the variable $z_3$) of the polynomial equation
\[
(z^t_{1ij})^m(\alpha_1(\lambda_j)+\lambda_j\alpha_2(\lambda_j))
+(z^t_{1ij})^{m-1}\big(\alpha_0(\lambda_j)+\alpha_3(\lambda_j)z_3\big)
-t(z_3^m+\epsilon_3)=0.
\]
After the change of variable $w=1/z_3$, this equation reduces to an equation in $w$; letting $t\to 0$ one obtains
\[
w^m+\frac{\alpha_3(\lambda_j)}{\alpha_0(\lambda_j)}w^{m-1}=0.
\]
Hence the limiting values of $w$ are $0$ and $-\alpha_3(\lambda_j)/\alpha_0(\lambda_j)$, with multiplicities $m-1$ and $1$, respectively.
Therefore, generically, $m^2$ isolated singularities of $Y_t$ converge to the curve $\W$ and $m^3-m^2$ of them converge
to the point at infinity
$
H_3\cap\W_0=p=[0:0:1:0],
$
where $H_3\subset\mathbb{P}^3$ is the hyperplane defined by $\xi_3=0$.
Moreover, for each $\lambda_j$ such that $\alpha_3(\lambda_j)=0$, there are $m$ additional singular points converging to $p$,
counted with multiplicity. Let $\beta$ be the maximal number of common roots of $\alpha_3(\lambda)=0$ and
$\epsilon_1 b(\lambda)-\epsilon_2a(\lambda)=0$ as $(\epsilon_1,\epsilon_2)$ vary. Thus, $\beta\leq m$.
Therefore, the minimum value for the limit \eqref{pertX} is  $m^2-\beta m.$

Let $\fol_t$ be the one-dimensional holomorphic foliation on $\mathbb{P}^3$ induced by $Y_t$.
On the hyperplane at infinity $H_3=\mathbb{P}^3\setminus\C^3=\mathbb{P}^2$, the foliation $\fol_t$ is described by the vector field
\[
\widetilde{Y_t}
=\bigl(t\alpha_2u_1-4u_1^3-tu_2^3+u_1g_t(u_1,u_2)\bigr)\frac{\partial}{\partial u_1}
+\bigl(t\alpha_2u_2-3u_1^3-2tu_2^3+u_2g_t(u_1,u_2)\bigr)\frac{\partial}{\partial u_2}.
\]
where
\[
g(u_1,u_2)=-t+\sum_{i=1}^{m-1}(\alpha_{i,1}u_1+\alpha_{i,2}u_2+\alpha_{i,3})u_1^{m-1-i}u_2^i,
\qquad
u_i=\xi_{i-1}/\xi_{2}\ \text{ for } i=1,2.
\]
A direct computation gives $\mu(\fol_t|_{H_3},0)\geq m^2$; in fact,
$\mu(\fol_t|_{H_3},0)=m^2+N(\fol,A_{\W_0})$. Therefore, $m^2+N(\fol,A_{\W_0})$ singular points in $H_3$ converge to $p$, and
\[
\mu(\fol,p):=\lim_{t\to0}\sum_{\lim_t q^t_{ijk}=p}\mu(\fol_t,q^t_{ijk})
=m^3-m^2+\beta m+m^2+N(\fol,A_{\W_0}).
\]
Hence $\mu(\fol,p)\leq m^3+\beta m+N(\fol,A_{\W_0})$.
By Theorem~\ref{theorem1}, we have
\[
\mu(\fol,\W_0)=-\nu(\fol,\W_0,\varphi_a)+N(\fol,A_{\W_0}),
\]
where $-\nu(\fol,\W_0,\varphi_a)=m^3+m^2$ since $\deg(\fol)=m$, $\chi(\W_0)=2$, $\deg(\W_0)=1$,
and $\ell=m_{\E_1}(\pi^*\fol)=m-1$.
Therefore,
\[
\mu(Y_t,\W)
=-\nu(\fol,\W_0,\varphi_a)+N(\fol,A_{\W_0})
-\lim_{t\to0}\sum_{\lim q^t_{ijk}=p}\mu(\fol_t,q^t_{ijk}).
\]
Since $\mu(Y_t,\W)\geq 0$, it follows that
\[
\mu(\fol,p)\leq -\nu(\fol,\W_0,\varphi_a)+N(\fol,A_{\W_0})
= m^3+m^2+m-1=MS(\fol,p)-1,
\]
where $MS(\fol,p)$ denotes Soares' bound. This upper value for $\mu(\fol,p)$ (and hence the minimum value for $\mu(X_t,\W)$)
is achieved, for instance, by the vector field
\[
X(z)=z_2^m\frac{\partial}{\partial z_1}+z_2^m\frac{\partial}{\partial z_2}+z_1^m\frac{\partial}{\partial z_3}.
\]
\end{example}

The second example has a curve component $\W$, an embedded singular point on $\W$, and further isolated singularities. A translation perturbation leaves a curve component in the projective singular locus. A second perturbation, accompanied by a deformation of the curve, has zero-dimensional singular locus; the limiting positions of its singular points determine the contribution at infinity.

\begin{example}
Let $X$ be the holomorphic vector field on $\C^3$ defined by
\begin{equation*}
\begin{array}{ll}
X=&\bigl(3z_1(z_2-z_1^2)+z_3-z_1^3\bigr)\frac{\partial}{\partial z_1}
+\bigl((z_1+5)(z_2-z_1^2)+2(z_3-z_1^3)\bigr)\frac{\partial}{\partial z_2}+\\
&+\bigl(z_2(z_2-z_1^2)+z_3-z_1^3\bigr)\frac{\partial}{\partial z_3}.
\end{array}
\end{equation*}
The singular set of $X$ contains the curve $\W\subset\C^3$ defined by
$
z_2-z_1^2=z_3-z_1^3=0
$
and the isolated point $A=(1,3,-5)$. There is also an embedded singular point
$P=(1,1,1)\in\sing(X)\cap\W$. Indeed, from the first two components of $X$ we obtain
$5(z_1-1)(z_2-z_1^2)=0$, and setting $z_1=1$ yields the points $A$ and $P$.
Let $\fol$ be the holomorphic foliation on $\mathbb{P}^3$ induced by $X$ on $\C^3$. Then
\[
\sing(\fol)=\W_0\cup \W_1\cup[1:1:1:1]\cup[1:3:-5:1]\cup[16:12:7:0],
\]
where $\W_0$ is the twisted cubic defined by
\[
\xi_1\xi_3-\xi_0^2=\xi_2\xi_3^2-\xi_0^3=\xi_0\xi_3-\xi_1\xi_2=0,
\]
while $\W_1$ is defined by $\xi_0=\xi_3=0$.
To compute the Milnor number $\mu(\fol,\W_0)$, we first consider the usual holomorphic deformation $X_t$ of $X$ given by
\begin{equation}\label{perb1}
X_t= X-t\bigg(\epsilon_1\frac{\partial}{\partial z_1}
+\epsilon_2\frac{\partial}{\partial z_2}
+\epsilon_3\frac{\partial}{\partial z_3}\bigg),
\end{equation}
where the $\epsilon_i$ are generic complex numbers. To determine $\sing(X_t)$, we use the first two components of $X_t$ and obtain
\[
5(z_1-1)\bigl(z_2-z_1^2\bigr)=(2\epsilon_1-\epsilon_2)t.
\]
There are two cases: $2\epsilon_1\neq\epsilon_2$ and $2\epsilon_1=\epsilon_2$.
If $2\epsilon_1\ne\epsilon_2$, then $\sing(X_t)$ contains three points, with two singular points converging to $\W$
and the third point converging to $A$.
If $2\epsilon_1=\epsilon_2$, then $\sing(X_t)$ contains only two points, with one point converging to $A$
and the other singular point converging to $P$. Therefore,
\[
\mu(X_t,\W)=\lim_{t\to0}\sum_{p_i^t\in\mathcal{A}_{\W}}\mu(X_t,p_i^t)\geq 1,
\]
and the minimum value is attained when $2\epsilon_1=\epsilon_2$.
However, $\W_1\subset\sing(\fol_t)$ for all $t\ne0$, where $\fol_t$ is the foliation on $\mathbb{P}^3$ induced by $X_t$.
Both $X$ and the curve $\W_0$ are therefore deformed. Let $\W_t$ be defined by
\[
\xi_1\xi_3-\xi_0^2=\xi_2\xi_3^2-\xi_0^3-t\xi_1^3=0.
\]
Thus, $\W_t$ is a smooth complete intersection for all $t\ne 0$.
Consider the holomorphic deformation
\begin{equation*}
\begin{array}{ll}
Y_t&=\bigl(3z_1f_1^t(z)+f_2^t(z)-t\epsilon_1\bigr)\frac{\partial}{\partial z_1}
+\bigl((z_1+5)f_1^t(z)+2f_2^t(z)-t\epsilon_2\bigr)\frac{\partial}{\partial z_2}\\
&+\bigl(z_2f_1^t(z)+f_2^t(z)-t(\epsilon_3+\alpha_0 z_1^3+\alpha_1 z_2^3+\alpha_2 z_3^3)\bigr)\frac{\partial}{\partial z_3},
\end{array}
\end{equation*}
where $f_1^t(z)=z_2-z_1^2$, $f_2^t(z)=z_3-z_1^3-tz_2^3$, and $\epsilon_i$ and $\alpha_i$ are nonzero generic complex numbers.
Let $\G_t$ be the holomorphic foliation on $\mathbb{P}^3$ induced by $Y_t$.
For a general choice of the coefficients and $t\neq0$, the singular set of $Y_t$ consists of $27$ isolated points, counted with multiplicity.
Let $z_i^t=(z_{1i}^t,z_{2i}^t,z_{3i}^t)\in\sing(Y_t)$. From the first two components of $Y_t$ we obtain again
\[
5(z_{1i}^t-1)\bigl(z_{2i}^t-(z_{1i}^t)^2\bigr)=(2\epsilon_1-\epsilon_2)t,
\]
which implies
\[
\lim_{t\to0}z_{1i}^t=1,
\quad\text{or}\quad
\lim_{t\to0}\bigl(z_{2i}^t-(z_{1i}^t)^2\bigr)=0.
\]
In the first case, $\lim_{t\to0}z_{2i}^t$ is either $1$ or $3$. In the second, $\lim_{t\to0}z_i^t\in\W_t$.
First, assume that \[2\epsilon_1-\epsilon_2=\alpha_0=\alpha_1=\alpha_2=0.\]
Then the singular set of $Y_t$ consists of two isolated points, one converging to $A$ and the other converging to $P$.
Thus,
\[
\mu(Y_t,\W)=\lim_{t\to0}\sum_{p_i^t\in\mathcal{A}_{\W}}\mu(Y_t,p_i^t)\geq 1.
\]
Suppose now that $2\epsilon_1\neq\epsilon_2$. Then
\[
z_{3i}^t=t\bigg(\frac{(3\epsilon_2-\epsilon_1)z_{1i}^t-5\epsilon_1}{5(z_{1i}^t-1)}\bigg).
\]
From the third component of $Y_t$ we obtain a polynomial of degree $27$ determining the variable $z_{1i}^t$.
Let
\[
\nu_{1i}^0=\lim_{t\to0}\frac{1}{z_{1i}^t},
\]
which satisfies
\[
\nu^9(\nu-1)^8\bigg[5(\epsilon_3-\epsilon_1)\nu^{10}
+(3\epsilon_3-\epsilon_1-5\epsilon_3)\nu^9
+(2\epsilon_2-\epsilon_1)\nu^8
+5(\nu-1)\sum_{i=0}^{2}\alpha_{i}\nu^{6-2i}\bigg]=0.
\]
Consequently, $\sing(Y_t)$ consists of $27$ isolated points, of which $26$ converge to $\W_0$, since $\mu(X,A)=1$.
Among the eight points $z_i^t$ such that $\lim_{t\to0}z_{1i}^t=1$, only one converges to $A$, and the other seven converge to $P=(1,1,1)\in\W$.
Moreover, all points $z_i^t\in\sing(Y_t)$ such that $\lim_{t\to0}u_{1i}^t=0$ converge to
$(\mathbb{P}^3\setminus\C^3)\cap \W_0=[0:0:1:0]=p$.
To determine $\mu(\fol,\W_0)$, restrict $\G_t$ to the hyperplane at infinity $H_3=\mathbb P^3\setminus\C^3$.
On $H_3$, the foliation $\G_t$ is described by
\[
\widetilde{Y_t}
=\bigl(t\alpha_2u_1-4u_1^3-tu_2^3+u_1g_t(u_1,u_2)\bigr)\frac{\partial}{\partial u_1}
+\bigl(t\alpha_2u_2-3u_1^3-2tu_2^3+u_2g_t(u_1,u_2)\bigr)\frac{\partial}{\partial u_2}.
\]
where $g_t(u_1,u_2)=u_1^2u_2+(1+t\alpha_1)u_1^3+t(1+\alpha_2)u_2^3$ and $u_i=\xi_{i-1}/\xi_{2}$ for $i=1,2$.
A direct calculation shows that $\sing(\widetilde{Y_t})$ consists of $13$ isolated points, counted with multiplicity.
To locate them, we write $u_{1i}^t=\lambda_t u_{2i}^t$ for $u_i^t=(u_{1i}^t,u_{2i}^t)\in \sing(\widetilde{Y_t})$.
Then
\[
\begin{cases}
3\lambda_t^4-4\lambda_t^3+2t\lambda_t-t=0,\\[2pt]
u_{2i}^t\Bigl[t\alpha_2-\bigl(3\lambda_t^3+2t\bigr)(u_{2i}^t)^2
+\bigl(\lambda_t^2+(1+t\alpha_1)\lambda_t^3+t(1+\alpha_2)\bigr)(u_{2i}^t)^3\Bigr]=0.
\end{cases}
\]
For every $\lambda_i^t$ solving the first equation ($i=1,2,3,4$), there are three isolated singular points different from the origin $O=(0,0)\in H_3$.
Letting $t\to0$ in the first equation yields
\[
3\lambda^4-4\lambda^3=0
\quad\Longleftrightarrow\quad
\lambda^3(3\lambda-4)=0.
\]
Hence, counting multiplicities, $\lambda=0$ is a triple root and $\lambda=4/3$ is a simple root.
The parametrisation $u_1=\lambda u_2$ covers only the singular points with $u_2\neq0$. The origin $O=(0,0)$ is considered separately; it belongs to $\sing(\widetilde{Y_t})$ for every $t$.
For $\lambda=4/3$, the second equation at $t=0$ becomes
\[
u_2\Big(-3\lambda^3 u_2^2+(\lambda^2+\lambda^3)u_2^3\Big)=0,
\]
so besides $u_2=0$ there is a nonzero solution
\[
u_2=\frac{3\lambda^3}{\lambda^2+\lambda^3}=\frac{3\lambda}{1+\lambda},
\qquad\text{and hence}\qquad
(u_1,u_2)=\Big(\frac{4}{3}\cdot\frac{12}{7},\frac{12}{7}\Big)=\Big(\frac{16}{7},\frac{12}{7}\Big).
\]
Thus at least one singular point converges to $(16/7,12/7)\in H_3$, rather than to $O$. The remaining singularities converge to $O$, with multiplicity, and give the stated local Milnor number there.
\end{example}

The final example concerns the fixed component $\W=\{z_1=z_2=0\}$. One class of perturbations has no isolated singularities near $\W$, so that $\mu(X_t,\W)=0$. Polynomial truncations of the transcendental coefficients, however, give projective foliations whose local Milnor number at $p=[0:0:1:0]$ is arbitrarily large. Consequently, there is no uniform upper bound for the contribution at infinity within this class.

\begin{example}
Consider the vector field on $\C^3$
\begin{equation}
X= \big( z_1\cos(z_3) + z_2\sin(z_3)\big)\frac{\partial}{\partial z_1}
+ \big( -z_1\sin(z_3) + z_2\cos( z_3)\big)\frac{\partial}{\partial z_2}
+ z_1P_m(z_3)\frac{\partial}{\partial z_3},
\end{equation}
where $P_m$ is a polynomial of degree $m$ with distinct roots $\gamma_i$ for $i=1,\ldots,m$.
Its singular set is $\W=\{z_1=z_2=0\}$. The component $\W$ is totally simple, since the Jacobian matrix $JX$ has an invertible $2\times2$ minor along $\W$.
Consider the small perturbations $X_t$ of $X$ given by
\begin{equation}
X_t = X-t\sum_{i=1}^{3}\big( a_{i0}+a_{i1}z_1+a_{i2}z_2\big)\frac{\partial}{\partial z_i},
\end{equation}
where the $a_{ij}$ are complex numbers. For $t\ne 0$, the singular set of $X_t$ is determined by solving the system
\begin{equation}\label{system1}
\left\{
\begin{array}{lcl}
z_1(\cos(z_3)-ta_{11})+z_2(\sin(z_3)-ta_{12})&=&ta_{10},\\
-z_1(\sin(z_3)+ta_{21})+z_2(\cos(z_3)-ta_{22})&=&ta_{20}
\end{array}
\right.
\end{equation}
for $(z_1,z_2)$ and then using the third component of $X_t$ to determine the $z_3$-coordinates of the isolated singularities.
For a general choice of the coefficients and $t\neq0$, the set $\sing(X_t)$ is an infinite discrete subset of $\mathcal A_{\W}$.
However, if $a_{30}\ne0$ and $a_{ij}=0$ for $i=1,2$, then $\sing(X_t)=\emptyset$.
Indeed, under this condition, if $p_t=(z_1^t,z_2^t,z_3^t)\in \sing(X_t)$, then $z_1^t=z_2^t=0$,
but the third component of $X_t$ cannot vanish. Hence
\[
\mu(X_t,\W) = \lim_{t\to0}\sum_{p_i^t\in\mathcal{A}_{\W}}\mu(X_t,p_i^t) \geq 0,
\]
and equality holds when $a_{30}\neq0$ and $a_{ij}=0$ for $i=1,2$. Consider next the polynomial deformation
\[
X_k^t = X_k - t\sum_{i=1}^{3}\big( a_{i0}+a_{i1}z_1+a_{i2}z_2\big)\frac{\partial}{\partial z_i},
\]
where
\[
X_k = \big( z_1f_{k}(z_3)+z_2g_k(z_3)\big)\frac{\partial}{\partial z_1}
+\big(-z_1 g_k(z_3)+f_k(z_3) z_2\big)\frac{\partial}{\partial z_2}
+z_1P_m(z_3)\frac{\partial}{\partial z_3},
\]
with
\[
f_k(z_3)=\sum_{i=0}^{k}(-1)^i\frac{z_3^{2i}}{(2i)!},
\qquad
g_k(z_3)=\sum_{i=0}^{k}(-1)^i\frac{z_3^{2i+1}}{(2i+1)!}.
\]
Let $\fol_k^t$ be the one-dimensional holomorphic foliation on $\mathbb{P}^3$ whose restriction to the affine chart $\C^3$
is described by $X_k^t$. The singular set of $X_k^t$ has at least $4k+2$ isolated singularities, counted with multiplicity.
For every $M>0$, there is $k_0\in\mathbb N$ such that $z_k^t\in\sing(X_k^t)$ and $k>k_0$ imply $|z_k^t|>M$. Consequently,
\[
\lim_{k\to\infty}\mu(\fol_k^t, p)=\infty,
\qquad
p=[0:0:1:0].
\]
The same construction applies in higher dimensions. Let $X=\sum_{i=1}^nP_i(z)\frac{\partial}{\partial z_i}$
be a vector field on $\C^n$ with $z=(z_1,\ldots,z_n)$, where
\[
P_i(z) = \sum_{j=1}^{d}z_jf_{ij}(z_{d+1},\ldots,z_n),\qquad i=1,\ldots,n,
\]
and each $f_{ij}=f_{ij}(z_{d+1},\ldots,z_n)$ is analytic. Thus,
$\W=\{z\in\C^n\mid z_1=\cdots=z_d=0\}\subset \sing(X).$
Assume that the matrix
\[
\mathbf{A}=\bigg[\frac{\partial P_i}{\partial z_j}\bigg]_{1\leq i,j\leq d}= \bigg[f_{ij}\bigg]
\]
is nonsingular for all $(z_{d+1},\ldots, z_n)\in\C^{n-d}$. Then there exists $\epsilon>0$ (small enough) such that the deformation
\[
X_t=X-t\sum_{i=d+1}^{n}a_{i}\frac{\partial}{\partial z_i},
\qquad a_i\ne0,
\]
has no isolated singular points converging to $\W$ for $0<|t|<\epsilon$.
\end{example}

\subsection*{Acknowledgements}
M. Corr\^ea is partially supported by the Universit\`a degli Studi di Bari and by PRIN 2022MWPMAB, ``Interactions between Geometric Structures and Function Theories''. He is a member of INdAM--GNSAGA.

\end{document}